\documentclass[12pt]{article}
\usepackage[a4paper,margin=2cm]{geometry}
\usepackage[colorlinks,citecolor=magenta,linkcolor=black]{hyperref}
\pdfpagewidth=\paperwidth \pdfpageheight=\paperheight
\usepackage{amsfonts,amssymb,amsthm,amsmath,eucal,tabu,url}
\usepackage{pgf}
 \usepackage{array}
 \usepackage{tikz-cd}
 \usepackage{pstricks}
 \usepackage{pstricks-add}
 \usepackage{pgf,tikz}
 \usetikzlibrary{automata}
 \usetikzlibrary{arrows}
 \usepackage{indentfirst}
\usepackage{tabularx} 

\usepackage[utf8]{inputenc}

\usepackage[colorinlistoftodos]{todonotes}
\theoremstyle{plain}
\newtheorem{thm}{Theorem}[section]
\newtheorem{theorem}[thm]{Theorem}
\newtheorem*{theoremA}{Theorem A}
\newtheorem*{theoremB}{Theorem B}

\newtheorem{proposition}[thm]{Proposition}
\newtheorem{corollary}[thm]{Corollary}

\theoremstyle{definition}
\newtheorem{definition}[thm]{Definition}
\newtheorem{remark}[thm]{Remark}

\newtheorem{thevarthm}[thm]{\varthmname}

\newenvironment{varthm*}[1]{\trivlist\item[]{\bf #1.}\it}{\endtrivlist}
\def\keywordname{{\bfseries Keywords}}%
\def\keywords#1{\par\addvspace\medskipamount{\rightskip=0pt plus1cm
\def\and{\ifhmode\unskip\nobreak\fi\ $\cdot$
}\noindent\keywordname\enspace\ignorespaces#1\par}}
\def\subclassname{{\bfseries Mathematics Subject Classification
(2020)}\enspace}
\def\subclass#1{\par\addvspace\medskipamount{\rightskip=0pt plus1cm
\def\and{\ifhmode\unskip\nobreak\fi\ $\cdot$
}\noindent\subclassname\ignorespaces#1\par}}

\begin{document}
\title{On free line arrangements with double, triple and quadruple points}
\author{Marek Janasz and Izabela Le\'sniak}
\date{\today}
\maketitle

\thispagestyle{empty}
\begin{abstract}
We show that there are only finitely many combinatorial types of free real line arrangements with only double, triple and quadruple intersection points, and we enlist all admissible weak-combinatorics of them. Then we classify all real $M$-line arrangements. In particular, we show that real $M$-line arrangements are simplicial.
\subclass{14N20, 52C35, 32S22}
\end{abstract}

\section{Introduction}
In the present paper we study free arrangements of lines in the real projective plane  having double, triple and quadruple points. Our motivation comes from two related problems. First of all, the celebrated Terao's freeness conjecture predicts that the freeness of line arrangements in the complex projective plane is determined by the combinatorics (i.e., the intersection lattice). In order to understand Terao's freeness conjecture one needs to check whether there exists a pair of line arrangements having the same combinatorics such that one is free and the second is not free. It is worth recalling here that Terao's freeness conjecture holds for arrangements with up to $d=14$ lines \cite{bara}. Another approach is to verify Terao's freeness conjecture by considering special classes of line arrangements. From the perspective of weak combinatorics, we can ask whether Terao's freeness conjecture holds for line arrangements with prescribed maximal multiplicity of intersection points. The first non-trivial situation is to consider the case of line arrangements with only double and triple intersections, and to approach this problem we would like to find some constraints on the data associated with the arrangement. Using a result due to Dimca and Sernesi \cite[Theorem 1.2]{DimSer} we can show that for a free arrangement $\mathcal{L} \subset \mathbb{P}^{2}_{\mathbb{C}}$ of $d$ lines we have the following inequality:
\begin{equation}
\label{sernesi}
\frac{d-1}{2} \geq \frac{2}{m(\mathcal{L})}\cdot d - 2,
\end{equation}
where $m(\mathcal{L})$ denotes the maximal multiplicity of intersection points in $\mathcal{L}$.
If $\mathcal{L}$ is a free arrangement with only double and triple points, then the above inequality gives us that $d\leq 9$, hence Terao's conjecture holds in this class of line arrangements. The next natural step is to focus on free line arrangements with double, triple and quadruple points. Let us observe that inequality \eqref{sernesi} \textit{does not give us any restriction on the degree} of such arrangement. However, using different methods we are able to show the following result in the class of real line arrangements  (working under the natural inclusion $\mathbb{R} \subset \mathbb{C}$).
\begin{theoremA}
If $\mathcal{L} \subset \mathbb{P}^{2}_{\mathbb{R}}$ is a free arrangement of $d$ lines having only double, triple and quadruple intersection points, then $d\leq 18$.
\end{theoremA}
Our result is a coarse estimate and it is very possible that we can decrease our bound by showing that certain weak combinatorics cannot be geometrically realized over the reals. Nevertheless this bound allows to present \textbf{all weak combinatorial types} of free real line arrangements with only double, triple and quadruple points, and this is exactly what we did in Theorem \ref{ThExists}. 

Then we focus on the so-called $M$-line arrangements which are defined as arrangements admitting double, triple and quadruple points that have the maximal possible total Milnor numbers for a given degree $d$, see Section 4 for details. Let us recall that for an arrangement $\mathcal{L} \subset \mathbb{P}^{2}$ of $d$ lines we define its weak combinatorics as the vector
$$W(\mathcal{L}) = (d;n_{2}, ..., n_{t}),$$
where $n_{i}$ denotes the number of $i$-fold intersection points and $t = m(\mathcal{L})$ denotes the maximal multiplicity among singular points of $\mathcal{L}$. Our main contribution towards this direction is the following classification result.
\begin{theoremB}
Let $\mathcal{L}\subset \mathbb{P}^{2}_{\mathbb{R}}$ be an $M$-arrangement of $d$ lines that is not a pencil, then $\mathcal{L}$ is simplicial and it has one of the following three weak combinatorics:
$$(d;n_{2},n_{3},n_{4}) \in \{(5;4,0,1),(9;6,4,3),(13;12,4,9)\}.$$
\end{theoremB}
In the last part of the paper, interesting on its own rights, we discuss some properties of algebraic surfaces associated to $M$-line arrangements. In particular, we provide some constraints on the Chern numbers of these surfaces.
\section{Preliminaries}
We follow the notation that was introduced in \cite{Dimca}. We denote by $S := \mathbb{C}[x,y,z]$ the coordinate ring of $\mathbb{P}^{2}_{\mathbb{C}}$. For a homogeneous polynomial $f \in S$, let $J_{f}$ denote the Jacobian ideal associated with $f$, i.e., the ideal of the form $J_{f} = \langle \partial_{x}\, f, \partial_{y} \, f, \partial_{z} \, f \rangle$. We will need an important invariant that is defined in the language of the syzygies of $J_{f}$.
\begin{definition}
Consider the graded $S$-module of Jacobian syzygies of $f$, namely $$AR(f)=\{(a,b,c)\in S^3 : a\partial_{x} \, f + b \partial_{y} \, f + c \partial_{z} \, f = 0 \}.$$
The minimal degree of non-trivial Jacobian relations for $f$ is defined to be 
$${\rm mdr}(f):=\min_{r\geq 0}\{AR(f)_r\neq 0\}.$$ 
\end{definition}
\begin{remark}
If $C = \{f=0\}$ is a reduced plane curve in $\mathbb{P}^{2}_{\mathbb{C}}$, then we write ${\rm mdr}(f)$ or ${\rm mdr}(C)$ interchangeably.
\end{remark}
Let us now formally define the freeness of a reduced plane curve.
\begin{definition}
A reduced curve $C \subset \mathbb{P}^{2}_{\mathbb{C}}$ of degree $d$ is free if the Jacobian ideal $J_{f}$ is saturated with respect to $\mathfrak{m} = \langle x,y,z\rangle$. Moreover, if $C$ is free, then the pair $(d_{1}, d_{2}) = ({\rm mdr}(f), d - 1 - {\rm mdr}(f))$ is called the exponents of $C$.
\end{definition}
In order to check whether a given plane curve $C$ is free we will use the following 
criterion~\cite{duP}.
\begin{theorem}[du-Plessis -- Wall]
\label{dup}
Let $C = \{f=0\}$ be a reduced plane curve of degree $d$ and let $r = {\rm mdr}(f)$. Let us denote the total Tjurina number of $C$ by $\tau(C)$.
Then the following two cases hold.
\begin{enumerate}
\item[a)] If $r < d/2$, then $\tau(C) \leq \tau_{max}(d,r)= (d-1)(d-r-1)+r^2$ and the equality holds if and only if the curve $C$ is free.
\item[b)] If $d/2 \leq r \leq d-1$, then
$\tau(C) \leq \tau_{max}(d,r)$,
where, in this case, we set
$$\tau_{max}(d,r)=(d-1)(d-r-1)+r^2- \binom{2r-d+2}{2}.$$
\end{enumerate}
\end{theorem}
\begin{remark}
Let us recall that for a reduced plane curve $C = \{f=0\}$ one has
$${\rm deg} \,J_{f} = \tau(C) = \sum_{p \in {\rm Sing}(C)} \mu_{p}$$
with $\mu_{p}$ being the Tjurina number of a singular point $p \in C$.
Moreover, since we only work with line arrangements, the Tjurina numbers of singular points are equal to the Milnor numbers. Hence, we have the following:
$$\tau(C) := \mu(C) = \sum_{p \in {\rm Sing}(C)} \mu_{p} = \sum_{r\geq 2}(r-1)^{2}t_{r},$$
where $t_{r}$ denotes the number of $r$-fold intersection points, i.e., points in the plane where exactly $r$ lines meet.
\end{remark}
We will work with real line arrangements for which the following well-known inequality holds \cite{Melchior}. Let us recall that $\mathcal{L}\subset \mathbb{P}^{2}_{\mathbb{R}}$ is called simplicial if all connected components of the complement $M(\mathcal{L}) = \mathbb{P}^{2}_{\mathbb{R}} \setminus \bigcup_{H \in \mathcal{L}}H$ are open $2$-simplices.
\begin{theorem}[Melchior]
Let $\mathcal{L} \subset \mathbb{P}^{2}_{\mathbb{R}}$ be an arrangement of $d\geq 3$ lines that is not a pencil. Then one has
\begin{align}\label{Melch}
n_2\geq 3+\sum_{k\geq 4} (k-3)\cdot n_k.
\end{align}
Moreover, the equality holds if and only if is simplicial.
\end{theorem}

\section{Free line arrangements with singular points of at most quadruple intersection points}
We start with the following general result.
\begin{proposition}\label{EstimDLine}
Let $\mathcal{L} = \{f = 
0\}\subset\mathbb{P}^2_{\mathbb{C}}$ be a 
free line arrangement of $d$ 
lines with  at 
most quadruple intersection points. 
Then 
\begin{align}\label{EstimateD}
n_2+n_3\leq \bigg\lfloor\frac{3d-3}{2}\bigg\rfloor \:\:\:\text{and}\:\:\:  
\bigg\lceil\frac{d^2-10d+9}{12}\bigg\rceil\leq n_4.     
\end{align}
\end{proposition}
\begin{proof}
Recall that the following naive combinatorial count holds for $\mathcal{L}$:
$$\frac{d(d-1)}{2}=\binom{d}{2}=n_2+3n_3+6n_4.$$
Furthermore, the total Tjurina number of $\mathcal{L}$ is equal to
$$\tau(\mathcal{L}) = n_{2} + 4n_{3} + 9n_{4}.$$
Combining these two facts, we get
$$d(d-1)=\tau(\mathcal{L})+n_2+2n_3+3n_4.$$
Since $\mathcal{L}$ is free, by Theorem \ref{dup} we have
$$r^2-r(d-1)+(d-1)^2=\tau(\mathcal{L}),$$
where $r={\rm mdr}(f)$.
Observe that we can rewrite the above identity as
$$r^2-r(d-1)+(d-1)^2=d(d-1)-n_2-2n_3-3n_4,$$
and hence we get
$$r^2-r(d-1)-d+1+n_2+2n_3+3n_4=0.$$
We compute now the discriminant $\triangle_{r}$ for the above quadratic equation and we get
$$\Delta=d^2-2d+1+4d-4-4n_2-8n_3-12n_4.$$
The freeness of $\mathcal{L}$ implies
$$d^2+2d-3-4n_2-8n_3-12n_4 \geq 0.$$
Using again the naive combinatorial count we obtain
$$d^2+2d-3 \geq 4n_2+8n_3+12n_4 \geq 2(n_2+n_3)+d(d-1),$$
and hence
$$2(n_2+n_3) \leq 3d-3,$$
which gives us first estimate. Since $n_{2}+n_{3}$ is an integer, then we can take the flooring. 

Let us now focus on the lower bound on the number of quadruple intersections. 
Using the naive combinatorial count and our previous bound $n_{2}+n_{3} \leq (3d-3)/2$, we get

$$d^2-d=2n_2+6n_3+12n_4 \leq 6(n_2+n_3)+12n_4 \leq 9d-9+12n_4$$
and hence
$$n_4 \geq \frac{d^2-10d+9}{12}.$$
Since $n_{4} \in \mathbb{Z}_{+}$, we can take the ceiling. Obviously the above bound is non-trivial provided that $d>10$.
\end{proof}
Using the above constraints we can find weak combinatorics of our free line arrangements. We say that an arrangement of lines $\mathcal{L}$ is trivial if $\mathcal{L}$ is a pencil of lines. Recall that pencils of lines are free line arrangements and thus we can ignore them in our discussion as non-interesting cases. 
\newpage
\begin{theorem}\label{ThExists}
If $\mathcal{L}$ is a non-trivial free real line arrangement with at most quadruple intersections, then it can have one of the following weak combinatorics:
\begin{table}[ht]
\centering
\small
\begin{tabular}{|c||cccccc|}
\hline
{\rm deg}$(\mathcal{L})$ &&&$(n_2,n_3,n_4)$&&&\\ \hline\hline
3 &$(3,0,0)$&&&&&\\ \hline
4 &$(3,1,0)$&&&&&\\ \hline
5 &$(4,0,1)$&$(4,2,0)$&&&&\\ \hline
6 &$(3,4,0)$&$(6,1,1)$&&&&\\ \hline
7 &$(3,6,0)$&$(6,1,2)$&$(6,3,1)$&$(9,0,2)$&&\\ \hline
8 &$(4,6,1)$&$(7,1,3)$&$(7,3,2)$&$(10,0,3)$&&\\ \hline
9 &$(6,4,3)$&$(6,6,2)$&$(9,1,4)$&$(9,3,3)$& $(12,0,4)$&\\ \hline
10 &$(6,7,3)$ & $(9,0,6)$&$(9,2,5)$&$(9,4,4)$&$(12,1,5)$&\\ \hline
11 &$(7,8,4)$&$(10,1,7)$&$(10,3,6)$&$(10,5,5)$&$(13,0,7)$&$(13,2,6)$  \\ \hline
12 &$(9,7,6)$&$(12,0,9)$&$(12,2,8)$&$(12,4,7)$& $(15,1,8)$& \\ \hline
13 &$(12,4,9)$&$(12,6,8)$&$(15,1,10)$&$(15,3,9)$&$(18,0,10)$&\\ \hline
14 &$(13,6,10)$&$(16,1,12)$&$(16,3,11)$&$(19,0,12)$&&\\ \hline
15 &$(15,6,12)$&$(18,1,14)$&$(18,3,13)$&$(21,0,14)$&&\\ \hline
16 &$(18,4,15)$&$(21,1,16)$&&&&\\ \hline
17 &$(22,0,19)$&$(22,2,18)$&&&&\\ \hline
18 & $(24,1,21)$&&&&&\\ \hline
\end{tabular}
\end{table}
\end{theorem}

\begin{proof}
We will use Melchior's inequality applied to $\mathcal{L}$, which has the form
\begin{equation}
\label{Melch1}
n_2\geq 3 + n_4.
\end{equation}
Moreover, we will need the first estimate in (\ref{EstimateD}), namely
$$(\star) : \quad n_{2} + n_{3}\leq (3d-3)/2.$$ We have the following chain of inequalities:
\begin{align*}
\binom{d}{2} = n_{2} + 3n_{3}+6n_{4} \stackrel{{\rm Melchior}}{\leq}n_{2} + 3n_{4} + 6(n_{2}-3) \leq 7(n_{2}+n_{3})-18 \stackrel{(\star)}{\leq}\dfrac{21d-21}{2}-18,
\end{align*}
which gives us
\begin{align*}
(d-3)(d-19)\leq 0.
\end{align*}
Hence $d \in \{3,4, \ldots, 19\}$.
We can now enumerate all the weak combinatorial types of such line arrangements using our \verb}Python} code, which can be found in the Appendix at the end of the paper. Using the naive combinatorial count and the bounds obtained in Proposition \ref{EstimDLine}, Melchior's inequality \eqref{Melch1} and Shnurnikov's inequality \eqref{Shnurnikov}, we extract the weak combinatorial types of the arrangements presented in the core of our statement above.
\end{proof}
\begin{remark}
In the above result we have just listed all possible weak-combinatorics of potential free line arrangements. A separate and difficult question is whether we can represent these weak combinatorics geometrically over the real numbers. We do not solve this problem in the present paper, since it is a highly non-trivial problem, especially when the number of lines is larger than $7$. However, in the next section we focus on a certain subset of these weak combinatorics and we explain how to construct some of them geometrically over the reals.
\end{remark}
\begin{remark}
In \cite{Geis}, Geis showed, using very different techniques, that if $m(\mathcal{L})\leq 4$, then $d\leq 19$. The main difference is, however, that he did not enumerate and analyze all potentially admissible weak-combinatorics in that setting.
\end{remark}
\section{{\it M}-line arrangements}

Our plan in this section is to provide analogous estimates for {\it M}-line arrangements as obtained in Proposition \ref{EstimDLine}. First we need to recall what $M$-line arrangements are. We present a~special variant of general results devoted to $M$-curves suitable for the setting of our paper.
\begin{theorem}[{cf. \cite[Theorem 4.8]{JanLes}}]\label{JanLes2dOne}
Let $\mathcal{L} \subset \mathbb{P}^{2}_{\mathbb{C}}$ be an arrangement of odd degree $d=2m+1\geq 5$ lines admitting only double, triple and quadruple points. Then $\mathcal{L}$ is an $M$-arrangement if and only if 
\begin{equation}
\label{eq11}
n_{2} + 4n_{3}+9n_{4} = 3m^{2}+1. 
\end{equation}
\end{theorem}

\begin{theorem}[{cf. \cite[Theorem 4.3]{JanLes}}]\label{JanLes2d}
\label{Mce}
Let $\mathcal{L} \subset \mathbb{P}^{2}_{\mathbb{C}}$ be an arrangement of even degree $d=2m\geq 4$ lines admitting only double, triple and quadruple points. Then $\mathcal{L}$ is an $M$-arrangement if and only if 
\begin{equation}
\label{eq22}
n_{2} + 4n_{3}+9n_{4} = 3m^{2}-3m+3.
\end{equation}
\end{theorem}
\begin{definition}
An line arrangement with only double, tripe and quadruple points satisfying either \eqref{eq11} or \eqref{eq22} is called as an {\it M}-line arrangement.
\end{definition}
\begin{remark}
Let us recall that \cite[Theorem 4.3 and 4.8]{JanLes} tell us that $M$-line arrangements are free.
\end{remark}
Now we provide constraints on the weak-combinatorics of $M$-line arrangements.
\begin{proposition}
Let $\mathcal{L} = \{f = 0\}\subset\mathbb{P}^2_{\mathbb{C}}$ be an 
$M$-line arrangement. 
\begin{enumerate}
\item[a)] If $d=2m+1 \geq 5$, then 
\begin{equation}
\label{cm2}
n_2+n_3\leq 3m \:\:\:\text{and}\:\:\:  
\bigg\lceil\frac{3m^2-12m+1}{9}\bigg\rceil\leq n_4. 
\end{equation}
\item[b)] If $d=2m \geq 6$, then 
\begin{equation}
    n_2+n_3\leq \bigg\lfloor\frac{6m-3}{2}\bigg\rfloor \:\:\:\text{and}\:\:\:  
\bigg\lceil\frac{m^2-5m+3}{3}\bigg\rceil\leq n_4. 
\end{equation}
\end{enumerate}
\end{proposition}

\begin{proof} 
We start with case {\it a)}. For $M$-line arrangements of odd degree $d=2m+1$ one has
\[n_2+n_3\leq \bigg\lfloor\frac{3d-3}{2}\bigg\rfloor=\bigg\lfloor\frac{6m+3-3}{2}\bigg\rfloor=3m.\]
Since $\tau(\mathcal{L}) = 3m^{2}+1$, we have
$$3m^2+1=n_2+4n_3+9n_4 \leq 4(n_2+n_3)+9n_4 \leq 6d-6+9n_4=12m+6-6+9n_4,$$
and this gives us
$$n_4 \geq \frac{3m^2-12m+1}{9}.$$

Now we pass to case {\it b)}. For $M$-line arrangements in even degree $d=2m$ we have
\[n_2+n_3\leq \bigg\lfloor\frac{3d-3}{2}\bigg\rfloor=\bigg\lfloor\frac{6m-3}{2}\bigg\rfloor .\]
Furthermore, since $\tau(\mathcal{L}) = 3m^{2}-3m+3$ we obtain
$$3m^2-3m+3=n_2+4n_3+9n_4 \leq 9n_4+4(n_2+n_3) \leq 9n_4 + 6d-6=9n_4+12m-6,$$
which gives us
$$n_4 \geq \frac{m^2-5m+3}{3}.$$
\end{proof}
\begin{proposition}
The following weak combinatorial types are admissible for $M$-line arrangements over the reals:
\begin{align*}
(d;n_{2},n_{3},n_{4}) \in & \{(5;4,0,1), (7;6,1,2), (9;6,4,3), (9;9,1,4), (10;9,0,6), (11;10,3,6), \\ & (11;13,0,7), (12;12,0,9), (13;12,4,9), (13;15,1,10), (15;18,1,14), (17;22,0,19)\}.    
\end{align*}
\end{proposition}
\begin{proof}
Let $\mathcal{l}$ be an $M$-line arrangement of odd degree 
$d = 2m+1$. Using Theorem \ref{ThExists}, Theorem \ref{JanLes2dOne}, and the condition
\begin{align*} 
n_{2} + 4n_{3} + 9n_{4} = 3m^2+1,
\end{align*}
we get the following admissible weak combinatorics
\begin{align*}
(d; n_2, n_3, n_4) \in & \{(5; 4, 0, 1),
 (7; 6, 1, 2), (9; 6, 4, 3),
 (9; 9, 1, 4), (11; 10, 3, 6), 
 (11; 13, 0, 7),\\ & (13; 12, 4, 9),
 (13; 15, 1, 10), (15; 18, 1, 14),
 (17; 22, 0, 19)\}.
\end{align*}
Repeating the reasoning for even degree $M$-line arrangements by taking into account Theorem \ref{ThExists}, Theorem \ref{JanLes2d}, and the condition
\begin{align*}
n_{2} + 4n_{3} + 9n_{4} = 3m^2-3m+3,
\end{align*}
we obtain the following admissible weak combinatorics
\begin{align*}(d; n_2, n_3, n_4) \in \{(10; 9, 0, 6),(12; 12, 0, 9)\}.
\end{align*}

\end{proof}

Now we would like to decide which weak combinatorial types can be geometrically realized over the real numbers.
\begin{theorem}
\label{classy}
The following weak combinatorics \textbf{cannot} be realized geometrically over the real numbers as line arrangements:
\begin{multline*}
(d;n_{2},n_{3},n_{4}) \in \{(7;6,1,2), (9;9,1,4), (10;9,0,6), (11;10,3,6), (11;13,0,7), (12;12,0,9), \\ (13;15,1,10), (15;18,1,14), (17;22,0,19)\}.    
\end{multline*}
\end{theorem}
\begin{proof}
In the first part of the proof we use a tool coming from the theory of pseudoline arrangements and hence it works also for real line arrangements. Let us recall that Shnurnikov in \cite{Sh} showed that if $\mathcal{L}\subset \mathbb{P}^{2}_{\mathbb{R}}$ is an arrangement of $d$ lines such that $n_{d}=n_{d-1}=n_{d-2}=0$ or $\mathcal{L}$ is not an arrangement of $d=7$ with $n_{4}=2$ and $n_{2}=9$, then one has
\begin{equation}
\label{Shnurnikov}
n_{2} + \frac{3}{2}n_{3} \geq 9 + \frac{1}{2}n_{4}.
\end{equation}
We can easy check that the weak combinatorics $(7;6,1,2)$, $(9;9,1,4)$, $(10;9,0,6)$ and $(12; 12,0,9)$ do not satisfy the above inequality and this implies that they cannot be geometrically realized with either straight or pseudolines.

To approach the cases $(11; 10,3,6)$, $(11;13,0,7)$ and $(13;15,1,10)$, we need to dive into the database constructed by Barakat and K\"uhne \cite{BK}. One can check there exist $8$ matroids having the weak combinatorics $(13;15,1,10)$, but they are not representable over any field, and there is no matroid having the weak combinatorics $(11;13,0,7)$. Finally, in the case of the weak combinatorics $(11; 10,3,6)$ there are exactly two matroids, one is not representable over any field, but the second one is representable. More precisely, such a line arrangement $\mathcal{L}$ is given by
\begin{multline*}
Q(x,y,z) = xyz(x+y)(x+z)(y-z)(x+(t+1)y)(x+(t+1)z)\cdot \\ (y+(-t-1)z)(x+(t+1)y-tz)(x+ty+z)    
\end{multline*}
subject to the condition that $t^{2}+t+1=0$. We can check that for all admissible $t$ the arrangement $\mathcal{L}$ is free, and this fact follows from the property that our arrangement is divisionally free, but we are not going to discuss this property here. However, $\mathcal{L}$ is not representable over the reals, and this completes our argument for this case.

Let us now pass to the weak combinatorics $(17;22,0,19)$. This weak combinatorics satisfies Melchior's inequality, namely
$$22 = n_{2} \geq 3 + n_{4} = 22,$$
hence if such an arrangement were geometrically realizable over the reals, then it would be a simplicial one. However, based on Cuntz \textit{et al.} classification of simplicial line arrangements \cite{Cuntz}, this is not the case and hence $(17;22,0,19)$ cannot be realized over the reals.

Finally, we come to weak combinatorics $(15;18,1,14)$. To approach this case we can use the theory of wiring diagrams, which decodes the existence of rank $3$ oriented matroids -- see \cite{CuntzS} for necessary definitions and algorithms regarding this subject. In particular, the non-existence of wiring diagrams with prescribed weak combinatorics implies that there is no real realization of this given weak-combinatorics by lines. Using a simple combinatorial script written in \verb}Python} (see Appendix) we can check that there does not exist any wiring diagram of $15$ wires producing $14$ quadruple intersection points, and hence there is no real realization of the weak combinatorics $(15;18,1,14)$ by lines.

This completes the proof.
\end{proof}
It is worth emphasizing that the weak combinatorics $(5;4,0,1)$, $(9;6,4,3)$ and $(13;12,4,9)$ are representable over the reals as \textbf{simplicial line arrangements}. The first one is easy to realize since this is just a near-pencil, the second one is the simplicial line arrangement $\mathcal{A}(9,1)$, and the last one is the simplicial line arrangement $\mathcal{A}(13,2)$. This observation allows us to conclude this part of the paper by the following intriguing result.
\begin{corollary}
If $\mathcal{L}$ is a non-trivial $M$-line arrangement in $\mathbb{P}^{2}_{\mathbb{R}}$, then it is simplicial.
\end{corollary}

\section{Algebraic surfaces associated with \textit{M}-line arrangements}
In this section we focus on algebraic surfaces associated with line arrangements. Let us briefly recall a general construction. Consider a pair $(\mathbb{P}^{2}_{\mathbb{C}},\mathcal{L})$, where $\mathcal{L}$ is an arrangement of $d\geq 6$ lines such that $n_{d}=0$. Denote by $\pi : X \rightarrow \mathbb{P}^{2}_{\mathbb{C}}$ the blowing up of the complex projective plane along all singular points of $\mathcal{L}$ with multiplicity $\geq 3$, and denote by $\widetilde{\mathcal{L}}$ the reduced total transformation of $\mathcal{L}$. Then $U_{\mathcal{L}} = X \setminus \widetilde{\mathcal{L}}$ is the log-surface associated with $(X, \widetilde{\mathcal{L}})$. Recall that the Chern numbers of $U_{\mathcal{L}}$ can be computed as follows:
\[c_{1}^{2}(U_{\mathcal{L}}) = 9 - 5d + \sum_{k\geq 2}(3k-4)n_{k},\]
\[c_{2}(U_{\mathcal{L}}) = 3 - 2d + \sum_{k\geq 2}(k-1)n_{k},\]
and for details regarding this subject please consult \cite[page 9]{BP}. 

Recall that from the general theory of algebraic log-surfaces the following inequality holds
$$\dfrac{c_{1}^{2}(U_{\mathcal{L}})}{c_{2}(U_{\mathcal{L}})} \leq 3,$$
and it follows from a deep logarithmic Miyaoka-Sakai inequality. It turns out that we can do better, namely that the following inequality holds
$$\dfrac{c_{1}^{2}(U_{\mathcal{L}})}{c_{2}(U_{\mathcal{L}})} \leq \frac{8}{3},$$
and we get equality if and only if $d=9$ and $n_{3}=12$, i.e. in the case of the famous dual Hesse arrangement of lines, and again please consult \cite{BP} for all necessary details regarding this subject. The main goal of this section is to establish bounds for ratios $c_{1}^{2}(U_{\mathcal{L}}) / c_{2}(U_{\mathcal{L}})$ of a different nature in the setting of $M$-line arrangements. Our result is somewhat surprising, as it shows that a complicated proof using a highly nontrivial tool in differential geometry provided by Miyaoka and Sakai can be replaced by the algebra-combinatorial definition of $M$-line arrangements.
\begin{theorem}
Let $\mathcal{L} \subset \mathbb{P}^{2}_{\mathbb{C}}$ be an $M$-arrangement of $d$ lines.
\begin{enumerate}
\item[a)] For odd $d\geq 9$ one has
$$\dfrac{c_{1}^{2}(U_{\mathcal{L}})}{c_{2}(U_{\mathcal{L}})} \leq \frac{11}{4}.$$
\item[b)] For even $d\geq 10$ one has
$$\dfrac{c_{1}^{2}(U_{\mathcal{L}})}{c_{2}(U_{\mathcal{L}})} \leq \frac{14}{5}.$$
\end{enumerate}
\end{theorem}
\begin{proof}
We start with the case $d=2m+1\geq 9$. Recall that in that situation one has
$$n_{2} + 2n_{3} + 3n_{4} = m^{2}+2m-1.$$
We start by computing the second Chern number of $U_{\mathcal{L}}$, namely
$$c_{2}(U_{\mathcal{L}}) = 3-2\cdot(2m+1) +n_{2} + 2n_{3} + 3n_{4} = m^2-2m.$$
Then we want to find an effective upper bound on $c_{1}^{2}(U_{\mathcal{L}})$. To do this, we need the following naive observation, which follows from combinatorial count:
\begin{equation}
\label{naive}
3(n_{3}+2n_{4}) \leq n_{2} + 3n_{3} + 6n_{4} = \binom{d}{2}.
\end{equation}
Then
\begin{multline*}
c_{1}^{2}(U_{\mathcal{L}}) = 9 - 5\cdot(2m+1) + 2n_{2} + 5n_{3} + 8n_{4} = 9-10m-5 + 2(n_{2}+2n_{3}+3n_{4}) + n_{3}+2n_{4} \\ \leq9-10m-5+2(m^{2}+2m-1)+\frac{m(2m+1)}{3} = \frac{1}{3}\cdot \bigg(8m^{2}-17m+6\bigg).  
\end{multline*}
We can easily check that 
$$\dfrac{c_{1}^{2}(U_{\mathcal{L}})}{c_{2}(U_{\mathcal{L}})} \leq \frac{8m^{2}-17m+6}{3m(m-2)} \leq \frac{11}{4}.$$

Now we pass to the second case with $d=2m\geq 10$. The proof goes along the same lines, so we present an outline. Since 
$$n_{2}+2n_{3}+3n_{4} = m^{2}+m-3$$
we have 
$$c_{2}(U_{\mathcal{L}}) = 3 - 4m + (m^{2}+m-3) = m^{2}-3m,$$
and 
\begin{multline*}
c_{1}^{2}(U_{\mathcal{L}}) = 9 - 10m+ 2n_{2} + 5n_{3} + 8n_{4} = 9-10m + 2(n_{2}+2n_{3}+3n_{4}) + n_{3}+2n_{4} \\ \leq 9-10m + 2(m^{2}+m-3)+\frac{m(2m-1)}{3} = \frac{1}{3}\cdot \bigg(8m^{2}-25m+9\bigg), 
\end{multline*}
hence
$$\dfrac{c_{1}^{2}(U_{\mathcal{L}})}{c_{2}(U_{\mathcal{L}})} \leq \frac{8m^{2}-25m+9}{3m(m-3)} \leq \frac{14}{5},$$
and this completes the proof.
\end{proof}
\begin{remark}
The fact that for $M$ line arrangements the second Chern number $c_{2}(U_{\mathcal{L}})$ depends only on the number of lines is very surprising to us. Except very special cases of line arrangements, this number depends entirely on the weak combinatorics of a given arrangement $\mathcal{L}$, not just on the number of lines. Moreover, our bounds on the quotients $\frac{c_{1}^{2}(U_{\mathcal{L}})}{c_{2}(U_{\mathcal{L}})}$ are also pretty good. Consider the Klein arrangement $\mathcal{K}$ of $21$ lines with $n_{3}=28$ and $n_{4}=21$. It is known that $\mathcal{K}$ is an example of $M$-line arrangements \cite{JanLes}. We can check that
$$2.65= \dfrac{c_{1}^{2}(U_{\mathcal{K}})}{c_{2}(U_{\mathcal{K}})} \leq \frac{8m^{2}-17m+6}{3m(m-2)} = 2.65,$$
and therefore, in this case, our bound is sharp. 
\end{remark}

\section*{Acknowledgment}
We would like to thank Piotr Pokora for his guidance and for sharing his thoughts about the project.
We would like to thank also Lukas K\"uhne for his help with realization spaces of some line arrangements and to Michael Cuntz for his help with the last missing case in our Theorem \ref{classy} and for suggesting to use wiring diagrams.

Marek Janasz is supported by the National Science Centre (Poland) Sonata Bis Grant Number 
\textbf{2023/50/E/ST1/00025}. For the purpose of Open Access, the authors have applied a CC-BY public copyright licence to any Author Accepted Manuscript (AAM) version arising from this submission.
\newpage
\section*{Appendix}
The following script has been used in the proof of Theorem 
\ref{ThExists} in order to enumerate all weak combinatorial 
types of our line arrangements.

\begin{verbatim}
import math

results = []

for d in range(1, 26):
    target_value = (d * (d - 1)) / 2
    max_i_j_sum = math.floor((3 * d - 3) / 2)
    min_k_value = math.ceil((d**2 - 10 * d + 9) / 12)

    for i in range(1, 26):
        for j in range(0, 26):
            for k in range(0, 26):
                cond1 = (i + 3 * j + 6 * k) == target_value
                cond2 = (i + j) <= max_i_j_sum
                cond3 = k >= min_k_value
                cond4 = (8 + 0.5 * k) <= (i + 1.5 * j)
                cond5 = (k + 3) <= i

                if all([cond1, cond2, cond3, cond4, cond5]):
                    print(f"{d}: {i} {j} {k}")
                    results.append((i, j, k))    
\end{verbatim}

Here we present a \verb}Python} script that allows us to conclude the non-existence of the weak-combinatorics $(d;n_{2},n_{3}, n_{4}) = (15;18,1,14)$.
\begin{verbatim}
import itertools
 
mainSet = {1,2,3,4,5,6,7,8,9,10,11,12,13,14,15} # labels of lines
k=4 
# Four-element subsets of the 15-element set
subsets = list(itertools.combinations(mainSet, 4)) 

print(len(subsets))
 
#The loop below searches among sets of four elements that have 
at most one element in common.

for j in range(0,len(subsets)): 
 if subsets[j]!=0:
  for i in range(j+1,len(subsets)):
   if subsets[i]!=0 and 1<len(set(subsets[j]).intersection(set(subsets[i]))):
subsets[i]=0
s = list(set(subsets))
s.remove(0)
print(len(s))
s.sort()
print(s)  
\end{verbatim}

\vskip 0.5 cm
%***************************************************************************** % Addresses
\bigskip
Marek Janasz,
Department of Mathematics,
University of the National Education Commission Krakow,
Podchor\c a\.zych 2,
PL-30-084 Krak\'ow, Poland. \\
\nopagebreak
\textit{E-mail address:} \texttt{marek.janasz@uken.krakow.pl}

\bigskip
Izabela Le\'sniak,
Department of Mathematics,
University of the National Education Commission Krakow,
Podchor\c a\.zych 2,
PL-30-084 Krak\'ow, Poland. \\
\nopagebreak
\textit{E-mail address:} \texttt{izabela.lesniak@uken.krakow.pl}

\begin{thebibliography}{000}

\bibitem{arnold}
V. I. Arnold, Local normal forms of 
functions. \textit{Invent. Math.} 
\textbf{35}: 87 -- 109 (1976).

\bibitem{bara}
M. Barakat and L. K\"uhne,
Computing the nonfree locus of the moduli space of arrangements and Terao’s freeness conjecture. \textit{Math. Comput.} \textbf{92(341)}: 1431 -- 1452 (2023).

\bibitem{BK}
M. Barakat and L. K\"uhne, \verb}matroids_split_public} – a database collection for
rank 3 integrally split simple matroids, \url{https://homalg-project.github.io/pkg/MatroidGeneration} (2019).

\bibitem{CuntzS}
M. Cuntz, Simplicial arrangements with up to 27 lines. \textit{Discrete Comput. Geom.} \textbf{48(3)}: 682 -- 701 (2012).

\bibitem{Cuntz}
M. Cuntz, S. Elia, and J.-P. Labb\'e, 
Congruence normality of simplicial hyperplane arrangements via oriented matroids. \textit{Ann. Comb.} \textbf{26(1)}: 1 -- 85 (2022).

\bibitem{Dimca}
A. Dimca,  \textit{Hyperplane arrangements. An introduction}. Universitext. Cham: Springer (ISBN 978-3-319-56220-9/pbk; 978-3-319-56221-6/ebook). xii, 200 p. (2017).

\bibitem{DimSer}
A. Dimca and E. Sernesi, Syzygies and logarithmic vector fields along plane curves. (Syzygies et champs de vecteurs logarithmiques le long de courbes planes.) \textit{J. Éc. Polytech., Math.} \textbf{1}: 247 -- 267 (2014).

\bibitem{duP}
A. Du Plessis and C. T. C. Wall, Application of the theory of the discriminant to highly singular plane curves. \textit{Math. Proc. Camb. Philos. Soc.} \textbf{126(2)}: 259 -- 266 (1999).

\bibitem{Geis}
D. Geis, Combinatorics of free and simplicial line arrangements. \textbf{arXiv:1809.09362} (2018).

\bibitem{JanLes} M. Janasz and I. Leśniak, On the existence of maximizing curves of odd degrees. 
\textit{Proc. Amer. Math. Soc.}, \url{https://doi.org/10.1090/proc/17335} (2025).

\bibitem{Melchior}  E. Melchior. \"{U}ber Vielseite der Projektive Ebene. Deutsche Mathematik {\bf 5}: 461 – 475 (1941).

\bibitem{BP}
B. Naskrecki and P. Pokora, On the geography of log-surfaces. \textbf{arXiv:2412.14635} (2024).

\bibitem{Sh}
I. N. Shnurnikov, A $t_{k}$ inequality for arrangements of pseudolines. \textit{Discrete Comput. Geom.} \textbf{55(2)}: 284 -- 295 (2016).
\end{thebibliography}
\end{document}